\documentclass[12pt,reqno]{amsart}
\usepackage{amssymb}
\usepackage{amsmath}
\usepackage{amsthm}
\usepackage{times}
\usepackage{graphicx}

\usepackage{color}
\usepackage{amsthm}
\newtheorem{theorem}{Theorem}
\newtheorem{lemma}{Lemma}
\newtheorem{corollary}{Corollary}
\newtheorem{proposition}{Proposition}

\usepackage[curve,all]{xy}
 \xyoption{curve}
 \xyoption{color}
 \xyoption{line}
\xyoption{arc}

\theoremstyle{definition}

\newtheorem{example}{Example}

\theoremstyle{remark}
\newtheorem{remark}{Remark}

\numberwithin{equation}{section}

\newcommand{\calH}{\mathcal{H}}
\newcommand{\calI}{\mathcal{I}}

\newcommand{\calJ}{\mathcal{J}}

\newcommand{\calO}{\mathcal{O}}

\newcommand{\calR}{\mathcal{R}}

\newcommand{\calX}{\mathcal{X}}

\newcommand{\bbk}{\mathbf{k}}

\newcommand{\bbC}{\mathbb{C}}

\newcommand{\bbP}{\mathbb{P}}

\newcommand{\bbQ}{\mathbb{Q}}

\newcommand{\bbZ}{\mathbb{Z}}

\newcommand{\bfe}{\mathbf{e}}

\newcommand{\bff}{\mathbf{f}}

\newcommand{\sfD}{\mathsf{D}}

\newcommand{\balpha}{\boldsymbol{\alpha}}
\newcommand{\bmu}{\boldsymbol{\mu}}
\newcommand{\half}{\frac{1}{2}}

\newcommand{\Aut}{\textup{Aut}}

\newcommand{\Pic}{\textup{Pic}}

\newcommand{\Spec}{\textup{Spec}}

\newcommand{\et}{\acute{e}t}

\newcommand{\Lef}{\textup{Lef}}

\newcommand{\Num}{\textup{Num}}

\newcommand{\Or}{\textup{O}}

\newcommand{\cha}{\textup{char}}
\newcommand{\nt}{\textup{nt}}
\newcommand{\ct}{\textup{ct}}

\newcommand{\NS}{\textup{NS}}

\newcommand{\red}{\textup{red}}
\newcommand{\Tr}{\textup{Tr}}

\renewcommand\emptyset\varnothing

\newcommand{\beq}{\begin{equation}}
\newcommand{\eeq}{\end{equation}}

\begin{document}
\title[Automorphisms of   Enriques surfaces]{Numerical trivial automorphisms of   Enriques surfaces in arbitrary characteristic}

\author{Igor V. Dolgachev}
\address{Department of Mathematics, University of Michigan, 525 E. University Av., Ann Arbor, Mi, 49109, USA}
\email{idolga@umich.edu}

\dedicatory{To the memory of Torsten Ekedahl}

\begin{abstract} We extend to arbitrary characteristic some known results on automorphisms of complex Enriques surfaces that  act identically 
on the cohomology or the cohomology modulo torsion. 
\end{abstract}

\maketitle

\section{Introduction} 
Let $S$ be algebraic surface over an algebraically closed field $\Bbbk$ of characteristic $p\ge 0$.  An automorphism $\sigma$ of $S$ is called \emph{numerically trivial} (resp, \emph{cohomologically trivial}) if it acts trivially on $H_{\et}^2(S,\bbQ_\ell)$ (resp. $H_{\et}^2(S,\bbZ_\ell)$).   In the case when $S$ is an Enriques surface, the Chern class homomorphism $c_1:\Pic(S) \to H_{\et}^2(S,\bbZ_\ell)$ induces an isomorphism $\NS(S)\otimes \bbZ_\ell \cong  H_{\et}^2(S,\bbZ_\ell)$, where $\NS(S)$ is the N\'eron-Severi group of $S$ isomorphic to the Picard group $\Pic(S)$. Moreover, it is known that the torsion subgroup of $\NS(S)$ is generated by the canonical class $K_S$. Thus, an automorphism $\sigma$  is cohomologically (resp. numerically) trivial if and only if it acts identically on $\Pic(S)$ (resp. $\Num(S) = \Pic(S)/(K_S)$). Over the field of complex numbers, the classification of numerically trivial automorphisms can be found in \cite{MN}, \cite{Mukai}. We have 

\begin{theorem}\label{thmmukai} Assume $\Bbbk = \bbC$. The group  $\Aut(S)_{\ct}$ of cohomologically  trivial automorphisms is  cyclic of order $\le 2$. The group $\Aut(S)_{\nt}$ of numerically trivial automorphisms is  cyclic of order $2$ or $4$. 
\end{theorem}

The tools in the loc.cit. are transcendental and use the periods of the K3-covers of Enriques surfaces, so they do not extend to the case of positive characteristic.  
 
 Our main result is that Theorem \ref{thmmukai} is true in any characteristic. 
 
The author is grateful to S. Kond\={o}, J. Keum  and the referee for useful comments to the paper.    

\section{Generalities}
Recall that an Enriques surface $S$ is called \emph{classical} if $K_S\ne 0$. The opposite may happen only if $\cha(\Bbbk) = 2$. Enriques surfaces with this property are divided into two classes: $\bmu_2$-surfaces or $\balpha_2$-surfaces. They are distinguished by the property of the action of the Frobenius on $H^2(S,\calO_S)\cong \Bbbk$.  In the first case, the action is non-trivial, and in the second case it is trivial. They also 
differ by the structure of their Picard schemes. In the first case it is isomorphic to the group scheme $\bmu_2$, in the second case it is isomorphic to the group scheme $\balpha_2$. Obviously, if 
$S$ is not classical, then 
$\Aut(S)_{\nt} = \Aut(S)_{\ct}$.

It is known that the quadratic lattice  $\Num(S)$ of numerical equivalence divisor classes on $S$ is isomorphic to $ \Pic(S)/(K_S)$. It  is a unimodular even quadratic lattice of rank 10 and signature $(1,9)$. As such it must  be isomorphic to  the orthogonal sum $E_{10} = E_8\oplus U$, where $E_8$ is the unique negative definite  even unimodular lattice of rank 8 and $U$ is a hyperbolic plane over $\bbZ$. One can realize $E_{10}$ as a primitive  sublattice of the standard unimodular  odd hyperbolic lattice 
\beq\label{1.1}
\bbZ^{1,10} = \bbZ\bfe_0+\bbZ\bfe_1+\cdots +\bbZ\bfe_{10},
\eeq
where $\bfe_0^2 = 1, \bfe_i^2 = -1, i > 0, \bfe_i\cdot \bfe_j = 0,\  i\ne j$. The orthogonal complement of the 
vector 
$$\bbk_{10}  = -3\bfe_0+\bfe_1+\cdots+\bfe_{10}$$
is isomorphic to the lattice $E_{10}$. 

Let 
$$\bff_j = -\bbk_{10}+\bfe_j,\  j = 1,\ldots, 10.$$
The 10 vectors $\bff_j$ satisfy 
$$\bff_j^2 = 0, \quad \bff_i\cdot \bff_j = 1, \ i\ne j. $$
Under an isomorphism $E_{10}\to \Num(S)$, their images form a sequence $(f_1,\ldots,f_{10})$ of isotropic vectors  satisfying  $f_i\cdot f_j = 1, i\ne j$, called an isotropic sequence in \cite{CD}.  An isotropic sequence  generates an index 3  sublattice of $\Num(S)$. 

A smooth rational curve $R$ on $S$ (a \emph{$(-2)$-curve}, for brevity) does not move in a linear system and $|R+K_S| = \emptyset$ if $K_S\ne 0$. Thus we can and will identify $R$ with its class $[R]$ in $\Num(S)$. Any $(-2)$-curve defines a reflection isometry of $\Num(S)$ 
$$s_R:x\mapsto x+(x\cdot R)R.$$ 

Any numerical divisor class  in $\Num(S)$ of non-negative norm represented by an effective divisor can be transformed by a sequence of reflections $s_R$ into the numerical divisor class of a nef divisor. Any isotropic sequence can be transformed by a sequence of reflections into a \emph{canonical} isotropic sequence, i.e. an isotropic sequence $(f_1,\ldots,f_{10})$  satisfying the  following properties
\begin{itemize} 
\item $f_{k_1},\ldots,f_{k_c}$ are nef classes for some $1 = k_1 < k_2< \ldots < k_c \le 10$;
\item $f_{j} = f_{k_i}+\calR_{j}, \ k_i<j < k_{i+1},$ where $\calR_j= R_{i,1}+\cdots+R_{i,s_j}$ is the sum of $s_j = j-k_i$ classes of $(-2)$-curves with intersection graph
of type $A_{s_j}$ such that $f_{k_i}\cdot \calR_j = f_{k_i}\cdot R_{i,1} = 1$ 

\xy (-50,10)*{};(-30,-5)*{};
@={(-10,0),(0,0),(10,0),(30,0),(40,0)}@@{*{\bullet}};
(-10,0)*{};(15,0)*{}**\dir{-};(25,0)*{};(40,0)*{}**\dir{-};
(20,0)*{\ldots};(0,3)*{R_{i,1}};(10,3)*{R_{i,2}};(30,3)*{R_{i,s_j-1}};(43,3)*{R_{i,s_j}};(-10,3)*{f_{k_i}}.
\endxy 
\end{itemize}

Any primitive isotropic numerical nef divisor class $f$ in $\Num(S)$ is the class of  nef effective divisors $F$ and  $F' \sim F+K_S$. The linear system $|2F| = |2F'|$ is  base-point-free and defines a  fibration $\phi:S\to \bbP^1$ whose generic fiber $S_\eta$ is a regular curve of arithmetic genus one. If $p\ne 2$, $S_\eta$ is a smooth elliptic curve over the residue field of the generic point $\eta$ of the base. In this case, 
$\phi$ is called an elliptic fibration. The divisors $F$ and $F'$ are \emph{half-fibers} of $\phi$, i.e. $2F$ and $2F'$ are fibers of $\phi$.

The following result by J.-P. Serre \cite{Serre} about lifting to characteristic 0  shows that there is nothing new if $p\ne 2$.

\begin{theorem} Let $W(\Bbbk)$ be the ring of Witt vectors with algebraically closed residue field $\Bbbk$, and  let $X$ be a 
smooth projective variety over $\Bbbk$, and let $G$ be a finite automorphism group of $X$. Assume 
\begin{itemize}
\item  $\#G$ is prime to $\cha(\Bbbk)$;
\item $H^2(X,\calO_X) = 0$;
\item $H^2(X,\Theta_X) = 0$, where $\Theta_X$ is the tangent sheaf of $X$.
\end{itemize}
Then the pair $(X,G)$ can be lifted to  $W(\Bbbk)$, i.e. there exists a smooth projective scheme 
$\calX\to \Spec~W(\Bbbk)$ with special fiber isomorphic to $X$ and an action of $G$ on $\calX$ over $W(\Bbbk)$ such 
that the induced action of $G$ in $X$ coincides with the action of $G$ on $X$.
\end{theorem}

We apply this theorem to the case when $G = \Aut_{\nt}(S)$, where $S$ is an Enriques surface over a field $\Bbbk$ 
of characteristic  $p \ne 2$. We will see later that  the order of $G = \Aut_{\nt}(S)$ is a power of 2, so it is prime to $p$.
 We have an isomorphism 
$H^2(S,\Theta_S) \cong H^0(S,\Omega_S^1(K_S))$. Let $\pi:X\to S$ be the K3-cover. Since  the  map
$\pi^*:H^0(S,\Omega_S^1(K_S))\to H^0(X,\Omega_X^1) \cong H^0(X,\Theta_X)$ 
is injective and $H^0(X,\Theta_X) = 0$, we obtain that all conditions in Serre's Theorem are satisfied. Thus, 
there is nothing new in this case. We can apply the results of \cite{MN} and \cite{Mukai} to obtain the complete 
classification of numerically trivial automorphisms. However, we will give here another, purely geometric,  proof of Theorem \ref{thmmukai} that does not appeal to K3-covers nor does it uses Serre's  lifting theorem.

\section{Lefschetz fixed-point formula}
We will need a Lefschetz fixed-point formula comparing the trace of an automorphism $\sigma$ of finite order acting on the $l$-adic cohomology 
$H_{\acute{e}t}^*(X,\bbQ_l)$ of a normal projective algebraic surface $X$ with the structure of the subscheme $X^\sigma$ of fixed points of $\sigma$. 

The subscheme of fixed points $X^\sigma$ is defined as the scheme-theoretical intersection of the diagonal with the graph of $\sigma$. Let $\calJ(\sigma)$ be the ideal sheaf of $X^\sigma$. If $x\in X^\sigma$, then the stalk $\calJ(\sigma)_x$ is the ideal in $\calO_{X,x}$ generated by elements $a-\sigma^*(a), a\in \calO_{X,x}$. Let $\Tr_i(\sigma)$ denote the trace of the linear action of $\sigma$ on $H_{\acute{e}t}^i(X,\bbQ_l)$. The following formula was proved in \cite{KSS}, Proposition (3.2):
\beq\label{lf1}
\sum (-1)^i\Tr_i(\sigma) = \chi(X,\calO_{X^\sigma})+\chi(X,\calJ(\sigma)/\calJ(\sigma)^2)-\chi(X,\Omega_X^1\otimes \calO_{X^\sigma}).
\eeq

If $\sigma$ is  \emph{tame}, i.e. its order is prime to $p$, then $X^\sigma$ is reduced and smooth \cite{Iversen}, and the Riemann-Roch formula easily implies
\beq\label{lf2}
\Lef(\sigma):= \sum (-1)^i\Tr_i(\sigma) = e(X^\sigma),
\eeq
where $e(X^\sigma)$ is the  Euler characteristic of $X^\sigma$ in \'etale $l$-adic cohomology. This is the familiar Lefschetz fixed-point formula from topology.

The interesting case is when $\sigma$ is \emph{wild}, i.e. its order is divisible by $p$. We will be interested in application of this formula in the case when 
$\sigma$ is of order 2 equal to the characteristic and $X$ is an Enriques surface $S$. 

Let $\pi:S\to Y = S/(\sigma)$ be the quotient map.  Consider an $\calO_Y$-linear map
$$T= 1+\sigma:\pi_*\calO_S \to \calO_Y.$$
Its image is the ideal sheaf $\calI_Z$ of a closed subscheme $Z$ of $Y$ and the inverse image of this ideal in $\calO_S$ is equal to $\calJ(\sigma)$.

\begin{theorem}\label{thm3.1} Let $S$ be a classical Enriques surface and let $\sigma$ be a wild automorphism of $S$ of order 2. 
Then $S^\sigma$ is non-empty and one of the following cases occurs:
\begin{itemize}
\item[(i)] $S^\sigma$ consists of one isolated point with $h^0(\calO_Z) = 1$.
\item[(ii)] $S^\sigma$ consists of  a connected curve with $h^0(\calO_Z)  = 1, h^1(\calO_Z) = 0$.
\end{itemize}  
Moreover, in both cases  $H^1(Y,\calO_Y) = H^2(Y,\calO_Y)  = \{0\}$.
\end{theorem}

\begin{proof} As was first observed by J.-P. Serre, the  first assertion follows from the Woods Hole Lefschetz fixed-point formula for cohomology with coefficients in a coherent sheaf \cite{SGA5} (we use that $\sum (-1)^i\Tr(g|H^i(S,\calO_S)) = 1$ and hence the right-hand side of the formula is not zero). The second assertion can be proved for a wild automorphism of any prime order by modifying the arguments from \cite{DK}. This was done in \cite{Keum}. Since in the case $p = 2$ the proof can be simplified, let  us reproduce it here.

Let $G$ be the cyclic group $(\sigma)$ generated by $\sigma$. Recall that one defines the cohomology sheaves $\calH^i(G,\calO_S)$ with stalks at closed points $x\in S$ isomorphic to the cohomology group $H^i(G,\calO_{S,x})$. Since $G$ is a cyclic group of order 2, the definition of the cohomology groups shows that  the sheaves $\calH^i(G,\calO_S), i = 1,2,$ coincide with the sheaf $\calO_Z$. It is isomorphic to the image of the trace map $T:\pi_*\calO_S \to \calO_Y$ and we have an exact sequence 
\beq\label{ex3}
0\to \calO_Y\to \pi_*\calO_S \to \calI_Z\to 0.
\eeq
It gives
\beq\label{odd}
2\chi(\calO_Y) = \chi(\calO_S)+\chi(\calO_Z) = 1+\chi(\calO_Z).
\eeq
Assume first that we have only isolated fixed points. By Lemma 2.1 from \cite{DK}, 
$$\omega_Y\cong (\pi_*\omega_S)^G,$$
hence, by Serre's Duality,
\beq\label{h2}
H^i(Y,\calO_Y) \cong H^{2-i}(Y,\omega_Y) \cong H^{2-i}(S,\omega_S)^G = H^i(S,\calO_S)^G = 0, \ i = 1,2.
\eeq
Since $\chi(\calO_Z) = h^0(\calO_Z) > 0$,  \eqref{ex3} gives $H^1(Y,\calO_Y)  = 0$ and $h^0(\calO_Z) = 1$. This is our case (i).

Assume now that $S^\sigma$, and hence $Z$, contains a one-dimensional part $Z_1$.  We use the following exact sequence 
$$0\to \omega_Y\to (\pi_*\omega_S)^G \to \mathcal{E}xt^1(\calH^2(G,\calO_S),\omega_Y) \to 0$$
from  Proposition 1.3 in \cite{DK}. It implies again  that $H^0(Y,\omega_Y) = H^2(Y,\calO_Y) = 0$.

Now, we use the exact sequence (2.3) from \cite{DK}
\beq\label{exx2}
0\to H^1(Y,\calO_Y)\to H^1 \to H^0(Y,\calH^1(G,\calO_S))\to H^2(Y,\calO_Y),
\eeq
where $H^1$ is a term in the exact sequence (2.1) from loc. cit.:
$$0\to H^1(G,H^0(S,\calO_S))\to H^1 \to H^0(G,H^1(S,\calO_S))\to H^2(G,H^0(S,\calO_S)).$$
Since $H^0(S,\calO_S)  \cong \Bbbk$ and $H^1(S,\calO_S) = 0$, we get  
$$H^1 \cong  \Bbbk.$$
Since $\dim H^0(Y,\calH^1(G,\calO_S)) = \dim H^0(Y,\calO_Z) > 0$, this gives 
$$H^1(Y,\calO_Y) = 0, \ \dim H^0(Y,\calO_Z) = 1.$$
Applying \eqref{odd}, we obtain $\chi(\calO_Z) = 1$, hence $h^1(\calO_Z) = 0$. This is our case (ii).
\end{proof}

\begin{remark} I do not know  whether the first case can occur. The second case occurs often, for example when the involution is the deck transformation of a superelliptic degree 2 separable map $S\to \sfD$, where $\sfD$ is a symmetric quartic  del Pezzo surface (see \cite{CD}). 

If $S$ is not classical surface, we have more possibilities. For example, the superelliptic separable map of a $\bmu_2$-surface gives an example of an involution with $Z$ equal to the union of an isolated fixed point and a connected curve of arithmetic genus 1.\footnote{The analysis of the fixed locus  in case of non-classical Enriques surfaces reveals a missing case in \cite{DK}:$X^\sigma$ may consist of an isolated fixed point and a connected curve.}  An example when $S^\sigma$ is an isolated fixed point or a connected rational curve is easy to construct. The first is obtained from  Example 2.8 in \cite{DK} by dividing the 
K3 surface by a fixed point free involution. The quotient is a $\mu_2$-surface. The second example is obtained from superelliptic map of an $\balpha_2$-surface.
\end{remark}

\begin{proposition}\label{prop}  
Let $S$ be a classical Enriques surface. Assume that $S^\sigma$ consists of one point $s_0$.  Then $\Lef(\sigma) = 4$. 
\end{proposition}

\begin{proof} Let $\pi:S\to Y = S/(\sigma)$ be the quotient morphism and $y = \pi(s_0)$. Since $h^0(\calO_Z) = 1$, it 
follows from \cite{Artin} that  the formal completion of the local ring 
$\calO_{Y,y}$ is a rational double point of type $D_4^{(1)}$ isomorphic to $\Bbbk[[x,y,z]]/(z^2+xyz+x^2y+xy^2)$ (see \cite{DK}, Remark 2.6). Moreover, identifying 
$\hat{\calO}_{Y,y}$ with the ring of invariants of $\hat{\calO}_{X,x_0}= \Bbbk[[u,v]]$, we have
$$x = u(u+y),  \ y = v(v+x), z= xu+yv.$$
This implies that the  ideal $\calJ(\sigma)_{s_0}$ generates the ideal $(u^2,v^2)$ in $\Bbbk[[u,v]]$. Applying \eqref{lf1},
 we easily obtain
$$\Lef(\sigma) = \dim_\Bbbk\Bbbk[[u,v]]/(u^2,v^2)+\dim_\Bbbk (u^2,v^2)/(u^4,v^4,u^2v^2)$$
$$-2\dim_\Bbbk\Bbbk[[u,v]]/(u^2,v^2) = 4+8-8 = 4.$$
\end{proof}

Since, for any $\sigma\in \Aut_{\ct}(S)$, we have  $\Lef(\sigma) = 12$, we obtain the following.

\begin{corollary}\label{corr} Let $\sigma$ be a wild numerically trivial automorphism of order 2 of a classical Enriques surface $S$. Then $S^\sigma$ is a connected curve. 
\end{corollary}

 \begin{proposition}\label{prop2} Let $S$ be a classical Enriques surface and $\sigma$ be a wild numerically trivial automorphism of order 2. Assume that $(S^\sigma)_{\red}$ is  contained in a fiber $F$ of a genus one fibration on $S$. 
Then $(S^\sigma)_\red = F_{\red}$ or consists of all irreducible components except one. 
 \end{proposition}

\begin{proof} Suppose $(S^\sigma)_{\red}\ne F_{\red}$. By the previous proposition, $S^\sigma$ is a connected part of $F$.  Since $\sigma$ fixes each irreducible component of $F$, and has one fixed point on each component which is not contained in $S^\sigma$, it is easy to see from the structure of fibres that $(S^\sigma)_{\red}$ consists of all components except one. 
\end{proof}

\section{Cohomologically trivial automorphisms}
 Let $\phi:S\to \bbP^1$ be a genus one fibration defined by a pencil $|2F|$. Let $D$ be an effective divisor on $S$. We denote by $D_\eta$ its restriction to the generic fiber $S_\eta$. If $D$ is of relative degree $d$ over the base of the fibration, then $D_\eta$ is an effective divisor of degree $d$ on $S_\eta$. In particular, if $D$ is irreducible, the divisor $D_\eta$ is a point on $S_\eta$ of degree $d$. Since $\phi$ has a double fiber, the minimal degree of a point on $S_\eta$ is equal to 2. 

\begin{lemma} $S$ admits a genus  one fibration  $\phi:S\to \bbP^1$ such that  $\sigma\in \Aut_{\ct}(S)$ leaves invariant all fibers 
of $\phi$ and at least 2 (3 if  $K_S\ne 0$) points of degree two on $S_\eta$. 
\end{lemma}

\begin{proof}   By Theorem 3.4.1 from \cite{CD} (we will  treat the exceptional case when $S$ is extra $E_8$-special in characteristic 2 in the last  section), one can find a canonical isotropic sequence $(f_1,\ldots,f_{10})$ with nef classes 
$f_1,f_{k_2},\ldots,f_{k_c}$ where $c\ge 2$. 

Assume $c \ge  3$. Then we have three genus one fibrations $|2F_1|, |2F_{k_2}|,$ and $ |2F_{k_3}|$ defined by $f_1,f_{k_2}, f_{k_3}$. The restriction of $F_{k_2}$ and $F_{k_3}$ to the general fiber $S_\eta$ of the genus one fibration defined by the pencil $|2F_1|$ are two degree 2 points. If $K_S\ne 0$, then the half-fibers $F_{k_2}'\in |F_{k_2}+K_S|$ and $F_{k_3}'\in |F_{k_3}+K_S|$ define two more degree two points.

Assume $c = 2$. Let $f_1 = [F_1], f_2 = [F_{k_2}]$. By definition of a canonical isotropic sequence, we have the following graph of irreducible curves

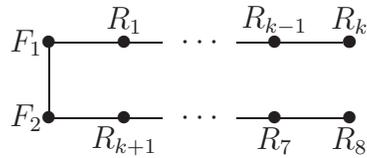
\xy (-30,10)*{};
@={(0,0),(10,0),(30,0),(40,0),(0,-10),(10,-10),(30,-10),(40,-10)}@@{*{\bullet}};
(0,0)*{};(15,0)*{}**\dir{-};(25,0)*{};(40,0)*{\bullet}**\dir{-};(0,-10)*{};,(15,-10)*{}**\dir{-};(25,-10)*{};(40,-10)*{}**\dir{-};
(0,-10)*{};(0,0)*{}**\dir{-};
(20,0)*{\cdots};(20,-10)*{\cdots};(-3,0)*{F_1};(10,3)*{R_1};(-3,-10)*{F_2};(30,3)*{R_{k-1}};(40,3)*{R_k};(10,-13)*{R_{k+1}};(30,-13)*{R_7};(40,-13)*{R_8};
\endxy

\medskip
Assume $k\ne 0$. Let $\phi:S\to \bbP^1$ be a genus one fibration defined by the pencil $|2F_1|$. Then the curves $F_2$ and $R_1$ define two points of degree two on $S_\eta$. If $S$ is classical, we have the third 
point defined by a curve $F_2'\in |F_2+K_S|$. Since $\sigma$ is cohomologically trivial, it leaves the half-fibers $F_1,  
F_2,$ and $F_2'$ invariant. It also leaves invariant the $(-2)$-curve $R_1$. If $k=0$, we take for  $\phi$ the fibration defined by the pencil $|2F_{k_2}|$ and get the same result.
\end{proof}

The next theorem extends the first assertion of Theorem \ref{thmmukai} from the Introduction to arbitrary characteristic. 

\begin{theorem}\label{prop1} 
The order of $\Aut_{\ct}(S)$ is equal to 1 or 2. 
\end{theorem}

\begin{proof}  By the previous Lemma, $\Aut_{\ct}(S)$ leaves invariant a genus one fibration and 2 or 3 degree two points on its generic 
fiber. For any $\sigma\in \Aut_{\ct}(S)$, the automorphism $\sigma^2$ acts identically on the residue fields of these points. 
If $p\ne 2$ (resp. $p = 2$), we obtain that $\sigma$, acting on the geometric generic fiber $S_{\bar{\eta}}$, fixes 6 (resp. 4) points. 
The known structure of the  automorphism group of an elliptic curve over an algebraically closed field of any characteristic  (see \cite{Silverman}, Appendix A) shows that this is possible only if $\sigma$ is the identity. 

So far, we have shown only that  each non-trivial element in $\Aut_{\ct}(S)$ is of order 2. However, the previous argument also shows that any two elements in the group share a common orbit in $S_{\bar{\eta}}$ of cardinality 2. Again, the known structure of the  
automorphism group of an elliptic curve shows that this implies that the group is of order 2.
\end{proof}

\begin{lemma}\label{lem4.3} Let $F$ be a singular fiber of a genus one fibration on an elliptic surface. Let $\sigma$ be a non-trivial 
tame automorphism of order 2 that leaves invariant  each  irreducible  component 
of $F$. Then 
\beq\label{lfred}
e(F^\sigma) = e(F).
\eeq 
\end{lemma}

\begin{remark} Formula  \eqref{lfred} agrees with the Lefschetz fixed-point formula whose proof in the case of a reducible curve I could not find.
\end{remark}

\begin{proof} The following pictures exhibit possible sets of fixed points. Here the star denotes an irreducible component in $F^\sigma$, the red line denotes the isolated fixed point that equal to  the intersection of two components, the red dot denotes an isolated fixed point which is not the intersection point of two components.

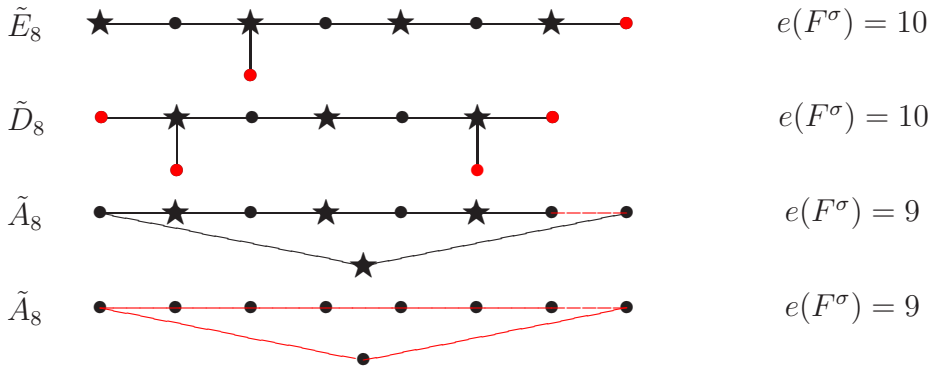
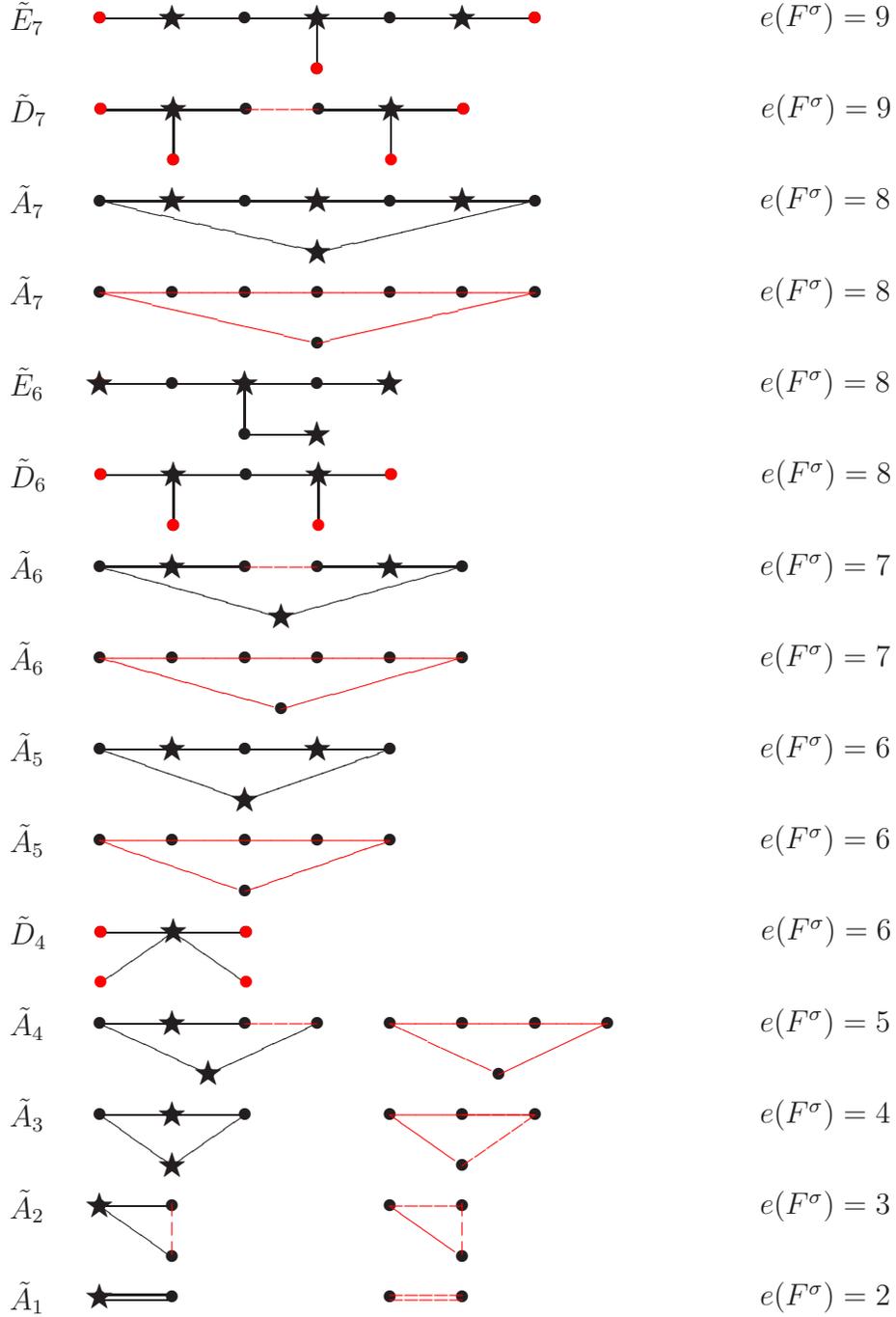
\xy (-10,10)*{};(-10,-10)*{};
@={(0,0),(10,0),(20,0),(30,0),(40,0),(50,0),(60,0),(70,0)}@@{*{\bullet}};
(0,0)*{};(70,0)*{}**\dir{-};(20,0)*{};(20,-7)*{\bullet}**\dir{-};
(0,0)*{\bigstar};(20,0)*{\bigstar};(40,0)*{\bigstar};(60,0)*{\bigstar};(20,-7)*{{\color{red}\bullet}};(70,0)*{{\color{red}\bullet}};
(100,0)*{e(F^\sigma) = 10};(-10,0)*{\tilde{E}_8};
\endxy

\xy (-10,0)*{};(-10,-10)*{};
@={(0,0),(10,0),(20,0),(30,0),(40,0),(50,0),(60,0)}@@{*{\bullet}};
(0,0)*{};(60,0)*{}**\dir{-};(10,0)*{};(10,-7)*{\bullet}**\dir{-};(50,0)*{};(50,-7)*{}**\dir{-};
(10,0)*{\bigstar};(30,0)*{\bigstar};(50,0)*{\bigstar};(60,0)*{{\color{red}\bullet}};(10,-7)*{{\color{red}\bullet}};(50,-7)*{{\color{red}\bullet}};
(100,0)*{e(F^\sigma) = 10};(-10,0)*{\tilde{D}_8};(0,0)*{{\color{red}\bullet}};
\endxy

\xy (-10,0)*{};(-10,-10)*{};
@={(0,0),(10,0),(20,0),(30,0),(40,0),(50,0),(60,0),(70,0)}@@{*{\bullet}};
(0,0)*{};(60,0)*{}**\dir{-};(0,0)*{};(35,-7)*{\bullet}**\dir{-};(35,-7)*{};(70,0)*{}**\dir{-};(60,0)*{};(70,0)*{}**{\color{red}\dir{-}};
(10,0)*{\bigstar};(30,0)*{\bigstar};(35,-7)*{\bigstar};(50,0)*{\bigstar};
(100,0)*{e(F^\sigma) = 9};(-10,0)*{\tilde{A}_8};
\endxy

\xy (-10,0)*{};(-10,-10)*{};
@={(0,0),(10,0),(20,0),(30,0),(40,0),(50,0),(60,0),(70,0)}@@{*{\bullet}};
(0,0)*{};(60,0)*{}**{\color{red}\dir{-}};(0,0)*{};(35,-7)*{\bullet}**{\color{red}\dir{-}};(35,-7)*{};(70,0)*{}**{\color{red}\dir{-}};(60,0)*{};(70,0)*{}**{\color{red}\dir{-}};
(100,0)*{e(F^\sigma) = 9};(-10,0)*{\tilde{A}_8};
\endxy

\xy (-10,0)*{};(-10,-10)*{};
@={(0,0),(10,0),(20,0),(30,0),(40,0),(50,0),(60,0)}@@{*{\bullet}};
(0,0)*{};(60,0)*{}**\dir{-};(30,0)*{};(30,-7)*{\bullet}**\dir{-};
(0,0)*{\color{red}{\bullet}};(10,0)*{\bigstar};(30,0)*{\bigstar};(50,0)*{\bigstar};(30,-7)*{{\color{red}\bullet}};(60,0)*{{\color{red}\bullet}};
(100,0)*{e(F^\sigma) = 9};(-10,0)*{\tilde{E}_7};
\endxy

\xy (-10,0)*{};(-10,-10)*{};
@={(0,0),(10,0),(20,0),(30,0),(40,0),(50,0)}@@{*{\bullet}};
(0,0)*{};(20,0)*{}**\dir{-};(30,0)*{};(50,0)*{}**\dir{-};(10,0)*{};(10,-7)*{\bullet}**\dir{-};(40,0)*{};(40,-7)*{}**\dir{-};
(10,0)*{\bigstar};(20,0)*{};(30,0)*{}**{\color{red}\dir{-}};(40,0)*{\bigstar};(50,0)*{{\color{red}\bullet}};(10,-7)*{{\color{red}\bullet}};(40,-7)*{{\color{red}\bullet}};
(100,0)*{e(F^\sigma) = 9};(-10,0)*{\tilde{D}_7};(0,0)*{{\color{red}\bullet}};
\endxy

\xy (-10,0)*{};(-10,-10)*{};
@={(0,0),(10,0),(20,0),(30,0),(40,0),(50,0),(60,0)}@@{*{\bullet}};
(0,0)*{};(60,0)*{}**\dir{-};(0,0)*{};(30,-7)*{\bullet}**\dir{-};(30,-7)*{};(60,0)*{}**\dir{-};
(10,0)*{\bigstar};(30,0)*{\bigstar};(30,-7)*{\bigstar};(50,0)*{\bigstar};
(100,0)*{e(F^\sigma) = 8};(-10,0)*{\tilde{A}_7};
\endxy

\xy (-10,0)*{};(-10,-10)*{};
@={(0,0),(10,0),(20,0),(30,0),(40,0),(50,0),(60,0)}@@{*{\bullet}};
(0,0)*{};(60,0)*{}**{\color{red}\dir{-}};(0,0)*{};(30,-7)*{\bullet}**{\color{red}\dir{-}};(30,-7)*{};(60,0)*{}**{\color{red}\dir{-}};(60,0)*{};(60,0)*{}**{\color{red}\dir{-}};
(100,0)*{e(F^\sigma) = 8};(-10,0)*{\tilde{A}_7};
\endxy

\xy (-10,0)*{};(-10,-10)*{};
@={(0,0),(10,0),(20,0),(30,0),(40,0)}@@{*{\bullet}};
(0,0)*{};(40,0)*{}**\dir{-};(20,0)*{};(20,-7)*{\bullet}**\dir{-};
(0,0)*{\color{red}{\bullet}};(20,0)*{\bigstar};(0,0)*{\bigstar};(40,0)*{\bigstar};(30,-7)*{\bigstar};(20,-7)*{};(30,-7)*{}**\dir{-};
(100,0)*{e(F^\sigma) = 8};(-10,0)*{\tilde{E}_6};
\endxy

\xy (-10,0)*{};(-10,-10)*{};
@={(0,0),(10,0),(20,0),(30,0),(40,0)}@@{*{\bullet}};
(0,0)*{};(40,0)*{}**\dir{-};(10,0)*{};(10,-7)*{\bullet}**\dir{-};(30,0)*{};(30,-7)*{}**\dir{-};
(10,0)*{\bigstar};(30,0)*{\bigstar};(40,0)*{{\color{red}\bullet}};(10,-7)*{{\color{red}\bullet}};(30,-7)*{{\color{red}\bullet}};
(100,0)*{e(F^\sigma) = 8};(-10,0)*{\tilde{D}_6};(0,0)*{{\color{red}\bullet}};
\endxy

\xy (-10,0)*{};(-10,-10)*{};
@={(0,0),(10,0),(20,0),(30,0),(40,0),(50,0)}@@{*{\bullet}};
(0,0)*{};(20,0)*{}**\dir{-};(30,0)*{};(50,0)*{}**\dir{-};(0,0)*{};(25,-7)*{\bullet}**\dir{-};(25,-7)*{};(50,0)*{}**\dir{-};
(10,0)*{\bigstar};(25,-7)*{\bigstar};(40,0)*{\bigstar};(20,0)*{};(30,0)*{}**{\color{red}\dir{-}};
(100,0)*{e(F^\sigma) = 7};(-10,0)*{\tilde{A}_6};
\endxy

\xy (-10,0)*{};(-10,-10)*{};
@={(0,0),(10,0),(20,0),(30,0),(40,0),(50,0)}@@{*{\bullet}};
(0,0)*{};(50,0)*{}**{\color{red}\dir{-}};(0,0)*{};(25,-7)*{\bullet}**{\color{red}\dir{-}};(25,-7)*{};(50,0)*{}**{\color{red}\dir{-}};(60,0)*{};(60,0)*{}**{\color{red}\dir{-}};
(100,0)*{e(F^\sigma) = 7};(-10,0)*{\tilde{A}_6};
\endxy

\xy (-10,0)*{};(-10,-10)*{};
@={(0,0),(10,0),(20,0),(30,0),(40,0)}@@{*{\bullet}};
(0,0)*{};(40,0)*{}**\dir{-};(0,0)*{};(20,-7)*{\bullet}**\dir{-};(20,-7)*{};(40,0)*{}**\dir{-};
(10,0)*{\bigstar};(20,-7)*{\bigstar};(30,0)*{\bigstar};
(100,0)*{e(F^\sigma) = 6};(-10,0)*{\tilde{A}_5};
\endxy

\xy (-10,0)*{};(-10,-10)*{};
@={(0,0),(10,0),(20,0),(30,0),(40,0)}@@{*{\bullet}};
(0,0)*{};(40,0)*{}**{\color{red}\dir{-}};(0,0)*{};(20,-7)*{\bullet}**{\color{red}\dir{-}};(20,-7)*{};(40,0)*{}**{\color{red}\dir{-}};(60,0)*{};(60,0)*{}**{\color{red}\dir{-}};
(100,0)*{e(F^\sigma) = 6};(-10,0)*{\tilde{A}_5};
\endxy

\xy (-10,0)*{};(-10,-10)*{};
@={(0,0),(10,0),(20,0)}@@{*{\bullet}};
(0,0)*{};(20,0)*{}**\dir{-};(0,-7)*{};(10,0)*{}**\dir{-};(10,0)*{};(20,-7)*{}**\dir{-};
(10,0)*{\bigstar};(0,0)*{{\color{red}\bullet}};(0,-7)*{{\color{red}\bullet}};(20,-7)*{{\color{red}\bullet}};
(100,0)*{e(F^\sigma) = 6};(-10,0)*{\tilde{D}_4};(20,0)*{{\color{red}\bullet}};
\endxy

\xy (-10,0)*{};(-10,-10)*{};
@={(0,0),(10,0),(20,0),(30,0)}@@{*{\bullet}};
(0,0)*{};(20,0)*{}**\dir{-};(0,0)*{};(15,-7)*{\bullet}**\dir{-};(15,-7)*{};(30,0)*{}**\dir{-};
(10,0)*{\bigstar};(15,-7)*{\bigstar};(20,0)*{};(30,0)*{}**{\color{red}\dir{-}};
(100,0)*{e(F^\sigma) = 5};(-10,0)*{\tilde{A}_4};
\xy (-10,0)*{};(-10,-10)*{};
@={(30,0),(40,0),(50,0),(60,0)}@@{*{\bullet}};
(30,0)*{};(60,0)*{}**{\color{red}\dir{-}};(30,0)*{};(45,-7)*{\bullet}**{\color{red}\dir{-}};(45,-7)*{};(60,0)*{}**{\color{red}\dir{-}};(60,0)*{};(60,0)*{}**{\color{red}\dir{-}};
\endxy
\endxy

\xy (-10,0)*{};(-10,-10)*{};
@={(0,0),(10,0),(20,0)}@@{*{\bullet}};
(0,0)*{};(20,0)*{}**\dir{-};(0,0)*{};(10,-7)*{}**\dir{-};(10,-7)*{};(20,0)*{}**\dir{-};
(10,0)*{\bigstar};(10,-7)*{\bigstar};
(100,0)*{e(F^\sigma) = 4};(-10,0)*{\tilde{A}_3};
\xy (-10,0)*{};(-10,-10)*{};
@={(30,0),(40,0),(50,0)}@@{*{\bullet}};
(30,0)*{};(50,0)*{}**{\color{red}\dir{-}};(30,0)*{};(40,-7)*{\bullet}**{\color{red}\dir{-}};(40,-7)*{};(50,0)*{}**{\color{red}\dir{-}};(40,0)*{};(50,0)*{}**{\color{red}\dir{-}};
\endxy
\endxy

\xy (-10,0)*{};(-10,-10)*{};
@={(0,0),(10,0),(10,-7)}@@{*{\bullet}};
(0,0)*{};(10,0)*{}**\dir{-};(0,0)*{};(10,-7)*{}**\dir{-};(10,-7)*{};(10,0)*{}**{\color{red}\dir{-}};
(0,0)*{\bigstar};
(100,0)*{e(F^\sigma) = 3};(-10,0)*{\tilde{A}_2};
\xy (-10,0)*{};(-10,-10)*{};
@={(30,0),(40,0)}@@{*{\bullet}};
(30,0)*{};(40,0)*{}**{\color{red}\dir{-}};(30,0)*{};(40,-7)*{\bullet}**{\color{red}\dir{-}};(40,-7)*{};(40,0)*{}**{\color{red}\dir{-}};
\endxy
\endxy

\xy (-10,0)*{};(-10,-10)*{};
@={(0,0),(10,0)}@@{*{\bullet}};
(0,0)*{};(10,0)*{}**\dir2{-};
(0,0)*{\bigstar};
(100,0)*{e(F^\sigma) = 2};(-10,0)*{\tilde{A}_1};
\xy (-10,0)*{};(-10,-10)*{};
@={(30,0),(40,0)}@@{*{\bullet}};
(30,0.35)*{};(40,0.35)*{}**{\color{red}\dir{-}};(30,-0.35)*{};(40,-0.35)*{}**{\color{red}\dir{-}};
\endxy
\endxy
Also, if $F$ is of type $\tilde{A}_2^*(IV)$ (resp. $\tilde{A}_2^*(III)$, resp. $\tilde{A}_1^*(II)$, resp. $\tilde{A}_0(I_1)$, resp. 
$\tilde{A}^{**}(II)$), we obtain that  $F^\sigma$ consists of 4 (resp. 3, resp. 2, resp. 1) isolated fixed points. Observe that the case $\tilde{D}_5$ is missing. It  does not occur. The equality $e(F^\sigma) = e(F)$ is checked case by case.

\end{proof}

\begin{theorem}\label{thm4.4} Assume that $K_S \ne 0$.  A  cohomologically trivial automorphism $\sigma$  leaves invariant any genus one  fibration  and acts identically on its base.
\end{theorem}

\begin{proof} The first assertion is obvious. Suppose $\sigma$ does not act identically on the base of a genus one fibration 
$\phi:S\to \bbP^1$. By assumption $K_S \ne 0$, hence  a genus one fibration has two half-fibers. Since $\sigma$ is cohomologically trivial, it fixes the two half-fibers $F_1$ and $F_2$ of $\phi$. Assume $p = 2$.  Since $\sigma$ acts on the base with only one fixed point, we get a contradiction. Assume $p\ne 2$. Then $\sigma$  has exactly two fixed points on the base. In particular, all non-multiple fibers must be irreducible, and the number of singular non-multiple fibers is even. 
 By Lefschetz fixed-point formula, we get
 $$e(S^\sigma) = e(F_1^\sigma)+e(F_2^\sigma) = 12.$$
 Since $p\ne 2$, $F_i$ is either smooth or of type $\tilde{A}_{n_i},i = 1,2 $. Suppose that $F_1$ and $F_2$ are singular fibers. Since $\sigma$ fixes any irreducible component of a fiber, Lemma \ref{lem4.3} implies that $e(F_i^\sigma) = e(F_i) = n_i$. So, we obtain that $n_1+n_2 = 12$. However, $F_1,F_2$ contribute $n_1+n_2-1$ to 
 the rank of the sublattice of $\Num(S)$ generated by components of fibers. The rank of this sublattice is at most $9$. This gives us a contradiction. Next we assume that one of the half-fibers is smooth. Then a smooth fiber has 4 fixed points, hence the other half-fiber must be of type $\tilde{A}_7$. It is easy to see that a smooth relatively minimal model of the quotient $S/(\sigma)$ has  singular fibers of type $\tilde{D}_4$ and $\tilde{A}_7$. Since the Euler characteristics of singular fibers add up to 12, this is impossible.

\end{proof}

\begin{remark} The assertion is probably true in the case when $S$ is not classical. However, I could prove only that $S$ admits at most one genus 1 fibration on which $\sigma$ does not act identically on the base. In this case $(S^\sigma)_\red$ is equal to the reduced half-fiber.

\end{remark}

We also have the converse assertion.
 
\begin{proposition}\label{enriques} Any numerically trivial automorphism $\sigma$ that  acts identically on the base of any genus one fibration is cohomologically trivial.
\end{proposition}

This follows from Enriques's Reducibility Lemma \cite{CD},  Corollary 3.2.2.  It asserts that any effective divisor on $S$ is linearly equivalent to a sum of irreducible curves of arithmetic genus one and smooth rational curves. Since each irreducible curve of arithmetic genus one is realized as either a fiber of a half-fiber of a genus one fibration, its class is fixed by $\sigma$. Since $\sigma$  fixes also the class of a smooth rational curve, we obtain that it acts identically on the Picard group.


\section{Numerically trivial automorphisms}
Here we will be interested in the group $\Aut_{\nt}(S)/\Aut_{\ct}(S)$. Since $\Num(S)$ coincides with  $ \Pic(S)$ for a non-classical Enriques surface $S$, we may assume that $K_S \ne 0$.

Let $\Or(\NS(S))$ be the group of automorphisms of the abelian group $\NS(S)$ preserving the intersection product. It follows from the elementary theory of abelian groups that 
$$\Or(\NS(S)) \cong (\bbZ/2\bbZ)^{10}\rtimes \Or(\Num(S)).$$
Thus 
\beq\label{elem}
\Aut_{\nt}(S)/\Aut_{\ct}(S)  \cong (\bbZ/2\bbZ)^a.
\eeq

The following theorem extends the second assertion of Theorem \ref{thmmukai} to arbitrary characteristic.

\begin{theorem}  
$$\Aut_{\nt}(S)/\Aut_{\ct}(S) \cong (\bbZ/2\bbZ)^a, \ a\le 1.$$
\end{theorem}

\begin{proof}

Assume first that $p\ne 2$.
Let $\sigma \in  \Aut_{\nt}(S)\setminus \Aut_{\ct}(S)$. By Proposition \ref{enriques}, there exists a genus one fibration $\phi:S\to \bbP^1$ such that  $\sigma$  acts non-trivially on its  base. Since $p\ne 2$, $\sigma$ has two fixed points on the base.  Let $F_1$ and $F_2$ be the fibers over these points. Obviously, $\sigma$ must leave invariant any reducible fiber, hence all fibers $F\ne F_1,F_2$ are irreducible. On the other hand, the Lefschetz fixed-point formula shows that one of the fixed fibers must be reducible. Let $G$ be the 
cyclic group generated by $(\sigma)$. Assume there is $\sigma'\in \Aut_{\nt}(S)\setminus G$. Since $\Aut_{nt}(S)/\Aut_{ct}(S)$ is an elementary $2$-group, the actions of $\sigma'$ and $\sigma$ on the base of the fibration commute. Thus $\sigma'$ either switches $F_1,F_2$  or it leaves them invariant. Since one of the fibers is reducible,  $\sigma'$ must fix both fibers.  We may assume that $F_1$ is reducible. By looking at all possible structure of the locus of fixed points containing in a fiber (see the proof of Lemma \ref{lem4.3}), we find that $\sigma$ and $\sigma'$ (or $\sigma\circ \sigma'$) fixes pointwisely the same set of irreducible components of $F_1$. Thus $\sigma\circ \sigma'$ (or $\sigma'$) acts identically on $F_1$. Since the set of fixed points is smooth, we get a contradiction with the assumption that $\sigma'\ne \sigma$.

Next we deal with the case $p = 2$. Suppose the assertion is not true. Let $\sigma_1,\sigma_2$ be two representatives of  non-trivial cosets 
in $\Aut_{\nt}(S)/\Aut_{\ct}(S)$. Let $\phi_i$ be a genus one fibration such that $\sigma_i$ does not act identically on its base.  Since, we are in characteristic 2, $\sigma_i$ has only one fixed point on the base. 
Let  $F_i$ be the unique fiber of $\phi_i$ fixed by $\sigma_i$. By Corollary \ref{corr},  $S^{\sigma_i}$ is connected curve $C_i$ contained in $(F_i)_{\red}$. Suppose $C_1$ is not contained in fibres of $\phi_2$ and $C_2$ is not contained in fibres of $\phi_1$. Then  a general fiber of $\phi_i$  intersects $C_j$, it has a fixed point on it. This implies that $\sigma_1$ acts identically on the base of $\phi_2$ and $\sigma_2$ acts identically on the base of $\phi_1$. Let $\sigma = \sigma_1\circ \sigma_2$. Then $\sigma\not \in \Aut_{ct}(S)$, hence repeating the argument, we obtain that $\sigma$ acts identically on all genus one fibrations except one. However, $ \sigma$ acts non-identically on the base of $\phi_1$ and on the base of $\phi_2$.  This contradiction proves the assertion. Now, if $C_1$ is contained in a fibre $F_2$ of $\phi_2$, then $\sigma_1$ leaves $F_2$ invariant, and, applying Proposition \ref{prop2}, we easily obtain that $S^{\sigma} = F_2$. Replacing $\sigma_1$ with $\sigma$, and repeating the argument we get the assertion.

\end{proof}

\section{Examples} In this section we assume that $p \ne 2$.

\begin{example} Let us see that the case when $\Aut_{\nt}(S)\ne \Aut_{\ct}(S)$ is realized. Consider $X = \bbP^1\times \bbP^1$ with two projections $p_1,p_2$ onto $\bbP^1$.   Choose two smooth rational curves $R$ and $R'$ of bidegree $(1,2)$ such that the restriction of $p_1$ to each of these curves is a finite map of degree two. Assume that $R$ is tangent to $R'$ at two points $x_1$ and $x_2$ with 
tangent directions corresponding to the fibers $L_1,L_2$ of $p_1$ passing through these points. Counting parameters, it is easy to see that this can be always achieved.  Let $x_1',x_2'$ be the points infinitely near $x_1, x_2$ corresponding to the tangent directions. Let $L_3,L_4$ be two fibers of $p_1$ different from $L_1$ and $L_2$. Let 
$$R \cap L_3 = \{x_{3},x_{4}\},\  R \cap L_4 = \{x_5,x_6\},\  R' \cap L_3 = \{x_{3}',x_{4}'\},\  R' \cap L_4 = \{x_5',x_6'\}.$$ 
We assume that all the points are distinct.  Let $b:X'\to X$ be the blow-up of the points $x_1,\ldots,x_6,x_1',\ldots,x_6'$.
Let $R_i,R_i'$ be the corresponding exceptional curves, $\bar{L}_i,\bar{R},\bar{R}'
$ be the proper transforms of $L_i, R,R'$. We have
$$D=\bar{R}+\bar{R}'+\sum_{i=1}^4\bar{L}_i+R_1+R_2 $$
$$\sim 2b^*(3f_1+2f_2)-2\sum_{i=1}^6(R_i+R_i')-4(R_1'+R_2'),$$
where $f_i$ is the divisor class of a fiber of the projection $p_i:X\to \bbP^1$. Since the divisor class of $D$ is divisible by 2 in the Picard group, we can construct a double cover  $\pi:S'\to X'$ branched over 
$D$. We have 
$$K_{X'} = b^*(-2f_1-2f_2)+\sum_{i=1}^6(R_i+R_i')+R_1'+R_2',$$
hence
$$K_{S'} = \pi^*(K_{X'}+\half D) =  (b\circ\pi)^*(f_1-R_1'-R_2').$$
We have $\bar{L}_1^2 = \bar{L}_2^2 = R_1^2 = R_2^2 = -2$, hence  
$\pi^*(\bar{L}_i) = 2A_i, i = 1,2,$ and  $ \pi^*(R_i) = 2B_i, i = 1,2,$
where 
$A_1,A_2,B_1,B_2$ are $(-1)$-curves. Also
$ \bar{R}^2 = \bar{R'}^2 = \bar{L}_3^2 = \bar{L}_4^2 = -4$, hence
$\pi^*(\bar{R}) = 2\tilde{R}, \pi^*(\bar{R'}) = 2\tilde{R'}, \pi^*(\bar{L}_3) = 2\tilde{L_3}, \pi^*(\bar{L}_4) = 2\tilde{L_4}$
where $\tilde{R},\ \tilde{R}',\  \tilde{L_3},\  \tilde{L_4}$ are $(-2)$-curves. The curves  $\bar{R}_i = \pi^*(R_i),\  \bar{R}_i' = \pi^*(R_i'), i = 3,4,5,6,$ are $(-2)$-curves. 
 The preimages of  the curves $R_1'$ and $R_2'$ are elliptic curves $F_1',F_2'$. Let
$\alpha:S'\to S$ be the blowing down of the curves $A_1,A_2,B_1,B_2$. Then the  preimage of the fibration $p_1:X\to \bbP^1$ on $S$ is an 
elliptic fibration with double fibers $2F_1,2F_2$, where $F_i = \alpha(F_i')$. 
 We have $K_S = 2F_1-F_1-F_2 = F_1-F_2$. So, $S$ is an Enriques surface with rational double cover $S\dasharrow \bbP^1\times \bbP^1$. The elliptic fibration has two fibers of types $\tilde{D}_4$ over $L_3,L_4$ and two double fibers over $L_1$ and $L_2$.

The following diagram pictures a configuration of curves on $S$.

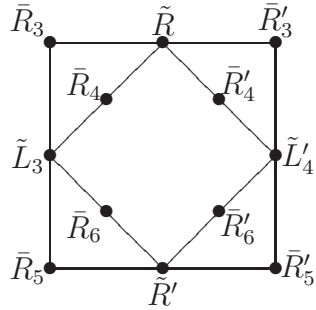
\xy (-40,25)*{};(-30,-30)*{};
@={(0,0),(0,15),(15,15),(30,15),(30,0),(30,-15), (15,-15),(0,-15),(7.5,7.5),(7.5,-7.5),(22.5,7.5),(22.5,-7.5)}@@{*{\bullet}};
(0,15)*{};(0,-15)*{}**\dir{-};(0,15)*{};(30,15)*{}**\dir{-};(30,15)*{};(30,-15)*{}**\dir{-};(30,-15)*{};(0,-15)*{}**\dir{-};
(0,0)*{};(15,15)*{}**\dir{-};(15,15)*{};(30,0)*{}**\dir{-};(30,0)*{};(15,-15)*{}**\dir{-};(15,-15)*{};(0,0)*{}**\dir{-};
(-3,0)*{\tilde{L}_3};(-3,18)*{\bar{R}_3};(15,18)*{\tilde{R}};(30,18)*{\bar{R}_3'};(33,0)*{\tilde{L}_4'};(4.5,9.5)*{\bar{R}_4};(25,9.5)*{\bar{R}_4'};(4.5,-9.5)*{\bar{R}_6};(15,-18)*{\tilde{R}'};(-3,-15)*{\bar{R}_5};(33,-15)*{\bar{R}_5'};(25,-9.5)*{\bar{R}_6'}
\endxy
Let us see that the cover automorphism is numerically trivial but not cohomologically trivial (see other treatment of this example  
in \cite{Mukai}). Consider the  pencil of curves of bidegree $(4,4)$ on $X$ generated by the curve $G = R+R'+L_3+L_4$ and $2C$, where $C$ is a unique curve of bidegree $(2,2)$ passing through the points $x_4,x_4',x_6,x_6',x_1,x_1',x_2,x_2'$.  These points are the double base points of the pencil. It is easy to see that this pencil defines an elliptic fibration  on $S$ with a double fiber of type $\tilde{A}_7$ formed by the curves 
$\bar{R}_3,\tilde{L}_3,\bar{R}_5,\tilde{R}',\bar{R}_5',\tilde{L}_4,\bar{R}_3',\tilde{R}$ and the double fiber $2\bar{C}$, where $\bar{C}$ is the preimage of $C$ on $S$. If $g= 0, f = 0$ are local equations of the curves $G$ and $C$, the local equation of a general member of the pencil is 
$g+\mu f^2 = 0$, and the local equation of the double cover $S\dasharrow X$ is $g = z^2$.  It clear that the pencil splits. By Proposition \ref{enriques} the automorphism is not cohomologically trivial.

Note that the K3-cover of $S$ has four singular fibers of type $\tilde{D}_4$. It is a Kummer surface of the product of two elliptic curves. This is the first example of a numerically trivial automorphism due to David  Lieberman (see \cite{MN}, Example 1). Over $\bbC$, a special case of this surface belongs to Kond\={o}'s list of complex Enriques surfaces with finite automorphism group \cite{Kondo}. It is a surface of type III. It admits five elliptic fibrations
of types
$$\tilde{D}_8,\  \tilde{D}_4+\tilde{D}_4, \ \tilde{D}_6+\tilde{A}_1+\tilde{A}_1,\ \tilde{A}_7+\tilde{A}_1, \tilde{A}_3+\tilde{A}_3+\tilde{A}_1+\tilde{A}_1.$$
\end{example}

\begin{example} Let $X = \bbP^1\times \bbP^1$ be as in the previous example.  Let $R'$ be a curve of bidegree $(3,4)$ on $X$ such that the degree of $p_1$ restricted to $R'$ is equal to 4. It is a curve of arithmetic genus 6. Choose three fibers of  $L_1,L_2,L_3$ of the first projection and points $x_i\in L_i$ on it no two of which lie on a fiber of the second projection. Let $x_i'\succ x_i$ be the point infinitely near $x_i$ in the tangent directions defined by the fiber $L_i$. We require that $R'$ has double points at $x_1,x_2,x_2',x_3,x_3'$ and a simple point at $x_1'$ (in particular $R'$ has a cusp at $x_1$ and has tacnodes at $x_2, x_3$).  The dimension of the linear 
system  of curves of bidegree $(3,4)$ is equal to $19$. We need 5 conditions  to have a cusp at $x_1$ as above, and 6 conditions for each tacnode. So, we 
can always find $R'$. 

Consider the double cover $\pi:Y \to X$ branched over $R'+L_1+L_2+L_3$. It has a double rational point of type $E_8$ over $x_1$ and simple 
elliptic singularities of degree 2 over $x_2,x_3$. Let $r:S'\to Y$ be a minimal resolution of singularities. The composition 
$f'= p_1\circ r\circ \pi:S'\to \bbP^1$ is a non-minimal  elliptic fibration on $S'$. It has a fiber $F_1'$ of type $\tilde{E}_8$ over $L_1$. The preimage of $L_2$ (resp. $L_2$) is the union of an elliptic curve $F_2'$ (resp. $F_3'$) and two disjoint $(-1)$-curves  $A_2,A_2'$ (resp. $A_3,A_3'$), all taken  with multiplicity $2$. Let $S'\to S$ be the blow-down of the curves $A_2,A_2',A_3,A_3'$. It is easy to check that $S$ is an Enriques surface with a fiber $F_1$ of type $\tilde{E}_8$ and two half-fibers $F_2, F_3$, the images of $F_2',F_3'$.

The following picture describes the incidence graph of irreducible components of $F_1$.

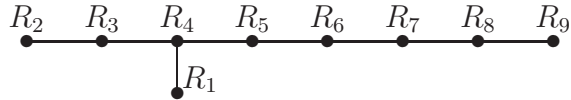
\xy (-30,10)*{};(-30,-10)*{};
@={(0,0),(10,0),(20,0),(30,0),(40,0),(50,0),(60,0),(70,0)}@@{*{\bullet}};
(0,0)*{};(70,0)*{}**\dir{-};(20,0)*{};(20,-7)*{\bullet}**\dir{-};
(23,-5)*{R_1};(0,3)*{R_2};(10,3)*{R_3};(20,3)*{R_4};(30,3)*{R_5};(40,3)*{R_6};(50,3)*{R_7};(60,3)*{R_8};(70,3)*{R_9};
\endxy

\noindent
Under the composition of rational maps $\pi:S\dasharrow S'\to Y\to X$, the image of 
the component $R_8$ is equal to $L_1$, the image of  the component $R_9$ is the intersection point $x_0\ne x_1$ of the curves $R'$ and  $L_1$. Let
$\sigma$ be the deck transformation of the cover $\pi$ (it extends to a biregular automorphism because $S$ is a minimal surface). 

Consider a curve $C$ on $X$ of bidegree $(1,2)$ that  passes through the points $x_1,x_2,x_2',x_3,x_3'$. The dimension of the linear system of curves of bidegree $(1,2)$ is equal to $5$. We have five condition for $C$ that  we can satisfy. The proper transform of $C$ on $S$ is a $(-2)$-curve $R_0$ that intersects the components $R_8$ and $R_2$. We have the  graph  which is contained in the incidence graph of $(-2)$-curves on $S$:

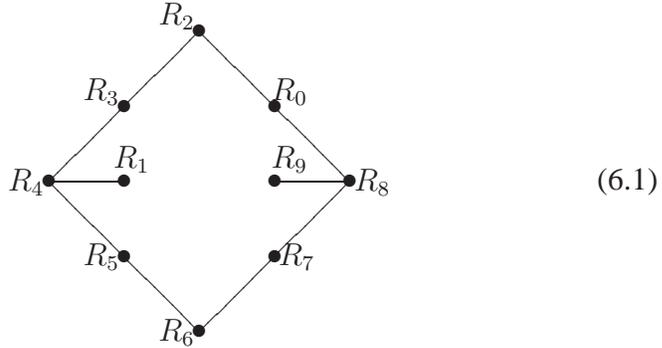
\beq\label{d1}
\xy (-10,25)*{};(-10,-30)*{};
@={(0,0),(10,10),(20,20),(30,10),(40,0),(30,-10), (20,-20),(10,-10),(10,0),(30,0)}@@{*{\bullet}};
(0,0)*{};(20,20)*{}**\dir{-};(20,20)*{};(40,0)*{}**\dir{-};(40,0)*{};(20,-20)*{}**\dir{-};(20,-20)*{};(20,-20),(0,0)*{}**\dir{-};
(0,0)*{};(10,0)*{}**\dir{-};(30,0)*{};(40,0)*{}**\dir{-};
(-3,0)*{R_4};(11,3)*{R_1};(7,12)*{R_3};(17,22)*{R_2};(7,-10)*{R_5};(17,-20)*{R_6};(33,-10)*{R_7};(43,0)*{R_8};(32,3)*{R_9};(32,12)*{R_0};
\endxy
\eeq
 
 One computes the determinant of the intersection matrix $(R_i\cdot R_j))$ and obtains that it is equal to $-4$. This shows that the curves $R_0,\ldots,R_9$ generate a sublattice  of index 2 of the lattice $\Num(S)$. The class of the half-fiber $F_2$ does not belong to this sublattice, but $2F_2$  belongs to it. This shows that the numerical classes $[F_2],[R_0],\ldots,[R_9]$ generate $\Num(S)$. We also have a section $s:\Num(S)\to \Pic(S)$ of the projection $\Pic(S)\to \Num(S) = \Pic(S)/(K_S)$ defined by sending $[R_i]$ to $R_i$ and $[F_2]$ to $F_2$. Since the divisor classes $R_i$ and $F_2$ are $\sigma$-invariant, we obtain that $\Pic(S) = K_S\oplus s(\Num(S))$, where the both summands are $\sigma$-invariant. This shows that $\sigma$ acts identically on $\Pic(S)$, and, by definition, belongs to $\Aut_{\ct}(S)$.
\end{example} 
 
\begin{remark} In fact, we have proven the following fact. Let $S$ be an Enriques surface such that  the incidence graph of $(-2)$-curves on it contains the  subgraph \eqref{d1}. Assume that $S$   admits an involution $\sigma$ that acts identically on the subgraph and leaves invariant the two half-fibers of the elliptic fibration defined by a subdiagram of type $\tilde{E}_8$. Then $\sigma\in \Aut_{\ct}(S)$.  The first example of such a pair $(S,\sigma)$ was constructed in \cite{Dolgachev}. The surface has additional $(-2)$-curves $R_1'$ and $R_9'$ forming the following graph.

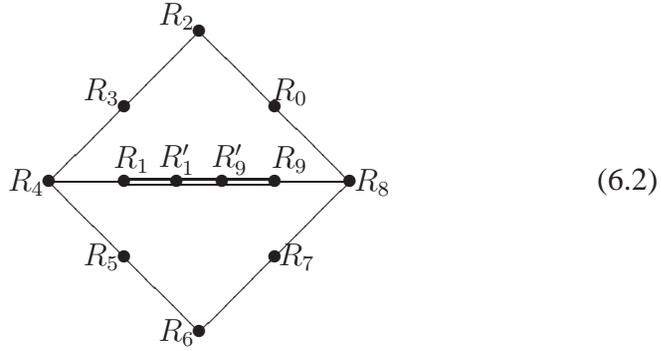
\beq\label{d2}
\xy (-10,25)*{};(-10,-30)*{};
@={(0,0),(10,10),(20,20),(30,10),(40,0),(30,-10), (20,-20),(10,-10),(10,0),(30,0)}@@{*{\bullet}};
(0,0)*{};(20,20)*{}**\dir{-};(20,20)*{};(40,0)*{}**\dir{-};(40,0)*{};(20,-20)*{}**\dir{-};(20,-20)*{};(20,-20),(0,0)*{}**\dir{-};
(0,0)*{};(10,0)*{}**\dir{-};(30,0)*{};(40,0)*{}**\dir{-};
(-3,0)*{R_4};(11,3)*{R_1};(17,3)*{R_1'};(7,12)*{R_3};(17,22)*{R_2};(7,-10)*{R_5};(17,-20)*{R_6};(33,-10)*{R_7};(43,0)*{R_8};(32,3)*{R_9};
(24,3)*{R_9'};(32,12)*{R_0};(10,0)*{};(17,0)*{}**\dir2{-};(23,0)*{};(30,0)*{}**\dir2{-};(17,0)*{};(23,0)*{}**\dir2{-};
(23,0)*{\bullet};(17,0)*{\bullet};
\endxy
\eeq
All smooth rational curves are accounted in this diagram. The surface has a finite automorphism group isomorphic to the dihedral group $D_4$ of order 4. It is a surface of type I in Kond\={o}'s list. The existence of an Enriques surface containing the diagram \eqref{d2} was first shown by E. Horikawa \cite{Horikawa}. Another construction of pairs $(S,\sigma)$ as above was given in \cite{MN} (the paper has  no reference to the paper \cite{Dolgachev} that had appeared in the previous issue of the same journal). 

 Observe now that in the  diagram \eqref{d1} the curves $R_0,\ldots,R_7$ form a nef isotropic effective divisor $F_0$ of type $\tilde{E}_7$. The curve $R_9$ does not intersect it. This implies that the genus one fibration defined by the pencil $|F_0|$ has a reducible fiber with one of its irreducible components equal to $R_9$. Since the sum of the Euler characteristics of fibers add up to $12$, we obtain that the fibration has a 
reducible fiber or a half-fiber of type $\tilde{A}_1$. Let $R_9'$ be its another irreducible component. Similarly, we consider the genus one fibration with fiber $R_0,R_2,R_3,R_5,\ldots,R_9$ of type $\tilde{E}_7$. It has another fiber (or a half-fiber) of type $\tilde{A}_1$ formed by $R_1$ and some other $(-2)$-curve $R_1'$. 

So any surface $S$ containing the configuration of curves from \eqref{d1} must contain a configuration of curves described by the following diagram.

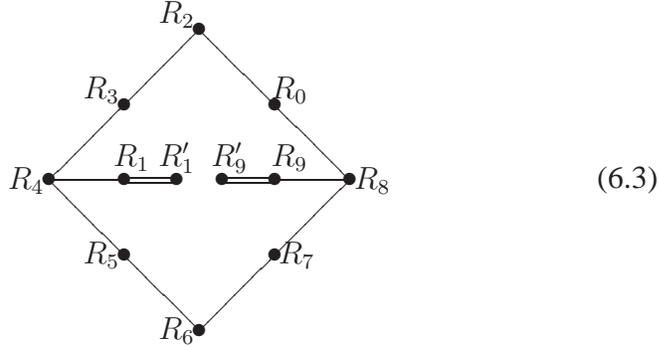
\beq\label{d3}
\xy (-10,25)*{};(-10,-30)*{};
@={(0,0),(10,10),(20,20),(30,10),(40,0),(30,-10), (20,-20),(10,-10),(10,0),(30,0)}@@{*{\bullet}};
(0,0)*{};(20,20)*{}**\dir{-};(20,20)*{};(40,0)*{}**\dir{-};(40,0)*{};(20,-20)*{}**\dir{-};(20,-20)*{};(20,-20),(0,0)*{}**\dir{-};
(0,0)*{};(10,0)*{}**\dir{-};(30,0)*{};(40,0)*{}**\dir{-};
(-3,0)*{R_4};(11,3)*{R_1};(17,3)*{R_1'};(7,12)*{R_3};(17,22)*{R_2};(7,-10)*{R_5};(17,-20)*{R_6};(33,-10)*{R_7};(43,0)*{R_8};(32,3)*{R_9};
(24,3)*{R_9'};(32,12)*{R_0};(10,0)*{};(17,0)*{}**\dir2{-};(23,0)*{};(30,0)*{}**\dir2{-};(23,0)*{\bullet};(17,0)*{\bullet};
\endxy
\eeq

Note that our  surfaces $S$ depend on 2 parameters. A general surface from the family is different from the Horikawa surface. For a general $S$, the curve $R_9'$ originates from a rational curve $Q$ of bidegree $(1,2)$ on $X$ that passes through the points $x_0$ and $x_2,x_2',x_3,x_3'$. It intersects $R_8$ with multiplicity 1. The curve $R_1'$ originates from a rational curve $Q'$ of bidegree $(5,6)$ of arithmetic genus 20 which has a 4-tuple point at $x_1$ and two double points infinitely near $x_1$. It also has four triple points at $x_2,x_2',x_3,x_3'$. It intersects $R_4$ with multiplicity 1. In the special case when one of the points $x_2$ or $x_3$ is contained in a  curve $(0,1)$  $Q_0$ of bidegree $(0,1)$ containing $x_0$, the curve $Q$ becomes reducible, its component $Q_0$ defines the curve $R_9'$ that   does not intersect $R_8$.  Moreover, if there exists a curve $Q_0'$ of bidegree $(2,3)$ which has  multiplicity 2 ar $x_2$, multiplicity  1 at $x_2',x_3,x_4$, and has a cusp at $x_1$ intersecting $R'$ at this point with multiplicity 7, then $Q_0'$ will define a curve $R_1'$ that does not intersect $R_4$. The two curves $R_1'$ and $R_9'$ will intersect at two points on the half-fibers of the elliptic fibration $|2F|$. This gives us the Horikawa surface.

\end{remark}

\begin{example} Let $\phi:X\to \bbP^1$ be a rational elliptic surface with reducible fiber $F_1$ of type IV and $F_2$ of type $I_0^* = \tilde{D}_4$ and one double fiber $2F$. The existence of such surface follows from the existence of a rational elliptic surface with a section with the same types of reducible fibers. Consider the double cover $X'\to X$ branched over $F_1$ and the union of the components of $F_2$ of multiplicity 1. It is easy to see that $X'$ is birationally equivalent to an Enriques surface with a fiber of type $\tilde{E}_6$ over $F_1$ and a smooth elliptic curve over $F_2$. The locus of fixed points of the deck transformation $\sigma$ consists of four components of the fiber of type $\tilde{E}_6$ and four isolated points on the smooth fiber. Thus the Lefschetz number is equal to 12 and $\sigma$ is numerically trivial. Over $\bbC$, this is Example 1 from \cite{Mukai} which was overlooked in \cite{MN}. A special case of this example can be found in \cite{Kondo}. It is realized on a surface of type V in Kond\={o}'s  list of Enriques surfaces with finite automorphism group.

\end{example}

\section{Extra special Enriques surfaces} In this section we will give  examples of cohomologically trivial automorphisms that appear only in characteristic 2.

An Enriques surface  is called \emph{extra special} if there exists a root basis  $B$ in $\Num(S)$ of cardinality $\le 11$ that consists of the classes of 
$(-2)$-curves such that the reflection 
subgroup $G$ generated by $B$  is of finite index in the orthogonal group of $\Num(S)$. Such a root basis was called 
\emph{crystallographic} in \cite{CD}. We additionally assume that no two curves intersect with multiplicity $> 2$. By a theorem of E. Vinberg \cite{Vinberg}, 
this is  possible if and only if the Coxeter diagram of the Coxeter group $(G,B)$ has the property that each affine subdiagram is 
contained in an affine diagram, not necessary connected, of maximal possible rank (in our case equal to 8).

One can easily classify extra special Enriques surfaces. They are of the following three kinds.

An extra $\tilde{E}_8$-special surface with the crystallographic basis of $(-2)$-curves described by the following diagram:

 
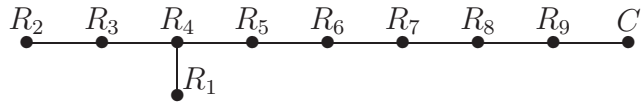
\xy (-10,10)*{};(-30,-10)*{};
@={(0,0),(10,0),(20,0),(30,0),(40,0),(50,0),(60,0),(70,0),(80,0)}@@{*{\bullet}};
(0,0)*{};(80,0)*{}**\dir{-};(20,0)*{};(20,-7)*{\bullet}**\dir{-};
(23,-5)*{R_1};(0,3)*{R_2};(10,3)*{R_3};(20,3)*{R_4};(30,3)*{R_5};(40,3)*{R_6};(50,3)*{R_7};(60,3)*{R_8};(70,3)*{R_9};(80,3)*{C};
\endxy

It has a genus one fibration with a half-fiber of type 
$\tilde{E}_8$ with irreducible components $R_1,\ldots,R_9$ and a smooth rational 2-section $C$.

An extra $\tilde{D}_8$-special surface with the crystallographic basis of $(-2)$-curves described by the following diagram:

 
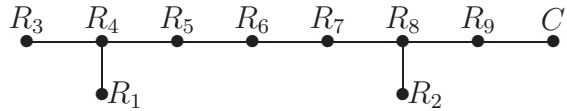
\xy (-10,10)*{};(-30,-10)*{};
@={(0,0),(10,0),(20,0),(30,0),(40,0),(50,0),(60,0),(10,-7),(50,-7),(70,0)}@@{*{\bullet}};
(0,0)*{};(70,0)*{}**\dir{-};(10,0)*{};(10,-7)*{}**\dir{-};(50,0)*{};(50,-7)*{}**\dir{-};
(13,-7)*{R_1};(0,3)*{R_3};(10,3)*{R_4};(20,3)*{R_5};(30,3)*{R_6};(50,3)*{R_8};(60,3)*{R_9};(70,3)*{C};
(40,3)*{R_7};(53,-7)*{R_2};
\endxy

It has a genus one fibration with a half-fiber of type $\tilde{D}_8$ with irreducible components $R_1,\ldots,R_9$  and  a 
smooth rational 2-section $C$.

An extra  $\tilde{E}_7+\tilde{A}_1$-special Enriques surface with the crystallographic basis of $(-2)$-curves described by 
the following diagram:

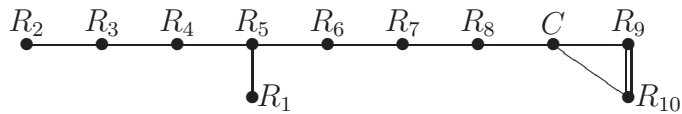
\xy (-10,10)*{};(-30,-10)*{};
@={(0,0),(10,0),(20,0),(30,0),(40,0),(50,0),(60,0),(70,0),(80,0), (80,-7),(30,-7)}@@{*{\bullet}};
(0,0)*{};(80,0)*{}**\dir{-};(30,0)*{};(30,-7)*{}**\dir{-};(80,0)*{};(80,-7)*{}**\dir2{-};(70,0)*{};(80,-7)*{}**\dir{-};
(33,-7)*{R_1};(0,3)*{R_2};(10,3)*{R_3};(20,3)*{R_4};(30,3)*{R_5};(40,3)*{R_6};(50,3)*{R_7};(60,3)*{R_8};(70,3)*{C};
(80,3)*{R_9};(84,-7)*{R_{10}};
\endxy

\noindent
or

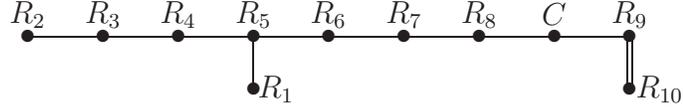
\xy (-10,10)*{};(-30,-15)*{};
@={(0,0),(10,0),(20,0),(30,0),(40,0),(50,0),(60,0),(70,0),(80,0), (80,-7),(30,-7)}@@{*{\bullet}};
(0,0)*{};(80,0)*{}**\dir{-};(30,0)*{};(30,-7)*{}**\dir{-};(80,0)*{};(80,-7)*{}**\dir2{-};(70,0)*{};
(33,-7)*{R_1};(0,3)*{R_2};(10,3)*{R_3};(20,3)*{R_4};(30,3)*{R_5};(40,3)*{R_6};(50,3)*{R_7};(60,3)*{R_8};(70,3)*{C};(80,3)*{R_9};(84,-7)*{R_{10}};
\endxy

 It has a genus one fibration with a half-fiber of type $\tilde{E}_7$ with irreducible components $R_1,\ldots,R_8$ and a 
fiber or a half-fiber   of type $\tilde{A}_1$ with irreducible components $R_9,R_{10}$. The curve $C$ is a  smooth rational 2-section.

\begin{remark} Note that the first two diagrams coincide with Coxeter diagrams of hyperbolic (in sense of Bourbaki) Coxeter groups of rank 10. 

One may ask why cannot the following  Coxeter diagrams  be realized.

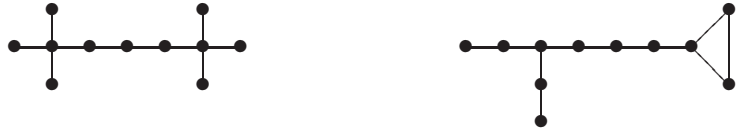
\xy
(-20,10)*{};(-20,-15)*{};
@={(0,0),(5,0),(10,0),(20,0),(25,0),(30,0),(5,5),(5,-5),(25,-5),(25,5)}@@{*{\bullet}};
(15,0)*{\bullet};
(0,0)*{};(30,0)*{}**\dir{-};(5,-5)*{};(5,5)*{}**\dir{-};(25,-5)*{};(25,5)*{}**\dir{-};
%
@={(60,0),(65,0),(70,0),(75,0),(80,0),(90,0),(70,-5),(70,-10),(95,5),(95,-5)}@@{*{\bullet}};
(85,0)*{\bullet};
(95,5)*{};(95,-5)*{}**\dir{-};(60,0)*{};(90,0)*{}**\dir{-};(90,0)*{};(95,5)*{}**\dir{-};(90,0)*{};(95,-5)*{}**\dir{-};(70,0)*{};(70,-10)*{}**\dir{-};
\endxy

In the first diagram one sees a genus 1 fibration with two half-fibers of type $\tilde{D}_4$.  Its jacobian fibration  is a genus 1 fibration on a rational surface.  By the Shioda-Tate formula,  the fibration is quasi-elliptic with the Mordell-Weil group isomorphic to 
$(\bbZ/2\bbZ)^2$. It acts by translation on our surface $S$ leaving the fibers invariant. In particular,  the central curve  and 
the half-fibers are fixed. It is easy to see that each automorphism has  two fixed points on each $(-2)$-curve, 
hence it acts identically on all curves. Since there are two half-fibers, we must have $K_S\ne 0$.  The locus of fixed points is a curve of arithmetic genus 1 contradicting Theorem \ref{thm3.1}.

In the second diagram, the Shioda-Tate formula implies that the Mordell-Weil group of the jacobian fibration is a cyclic group of order 3. It is easy to see, via the symmetry of the  graph, the surface does not admit an automorphism of order 3.

Also note that the group of automorphisms of an extra special Enriques surface is finite. I believe that together with Kond\={o}'s list these are all possible Enriques surfaces with finite automorphism group.
\end{remark}

It follows from the theory of reflection groups that the fundamental polyhedron for the Coxeter group $(G,B)$ in the 9-dimensional 
Lobachevsky  space is of finite volume. Its vertices at infinity correspond to maximal affine subdiagrams and also to 
$G$-orbits of primitive isotropic vectors in $\Num(S)$. The root basis $B$ is a maximal crystallographic basis, so the set of the curves 
$R_i,C$ is equal to the set of all $(-2)$-curves on the surface and the set of nef primitive isotropic vectors in $\Num(S)$ is equal to the set of 
 affine subdiagrams of maximal rank. Thus the number of genus one fibrations on $S$ is finite and coincides with the set of affine 
subdiagrams of rank 8.

It is not known whether an  extra $\tilde{D}_8$-special Enriques surface exists. However,  examples of extra-special surfaces of types $\tilde{E}_8$, or $\tilde{E}_7+\tilde{A}_1$ are given in \cite{Salomonsson}. They are either classical  Enriques surfaces or $\balpha_2$-surfaces. The surfaces are constructed as  separable double covers of a rational surface, so they always admit an automorphism $\sigma$ of order 2.

Suppose that $S$ is an extra $\tilde{E}_8$-special surface. Then we find that the surface has only one genus one fibration. It is
clear that $\sigma$ acts identically on the diagram. This allows one to define a $\sigma$-invariant splitting 
$\Pic(S) \cong \Num(S)\oplus K_S$. It implies that $\sigma$ is cohomologically trivial.

Assume that $S$ is  extra $\tilde{E}_7\oplus \tilde{A}_1$-special surface. The surface has a unique genus one fibration with a half-fiber of type $\tilde{E}_7$. It also has two fibrations in the first case and one fibration in the second case with a fiber of type $\tilde{E}_8$. It implies that the curves 
$R_1,\cdots,R_8,C$ are fixed under $\sigma$. It follows from Salomonsson's construction that  $\sigma(R_9)=R_{10}$ in the first 
case. In the second case, $R_9$ and $R_{10}$ are $\sigma$-invariant on any extra special $\tilde{E}_7\oplus \tilde{A}_1$-surface.

\bibliographystyle{plain}

  \end{document}